\documentclass[a4paper]{article}

\usepackage[english]{babel}
\usepackage{graphicx}
\usepackage{mathrsfs}
\usepackage{amsmath}
\usepackage{amsthm}
\usepackage{amssymb}
\usepackage{cleveref}
\usepackage{abstract}

\newtheorem{definition}{Definition}[section]
\newtheorem{proposition}[definition]{Proposition}
\newtheorem{observation}[definition]{Observation}
\newtheorem{remark}[definition]{Remark}
\newtheorem{lemma}[definition]{Lemma}
\newtheorem{corollary}[definition]{Corollary}
\newtheorem{theorem}[definition]{Theorem}
\newcommand{\nc}{\newcommand}
\nc{\store}{\rotatebox[origin=c]{90}{$\thickspace\eqslantless\thickspace$}}
\nc{\stor}{\thickspace\store\thickspace}
\nc{\ST}{\mathsf{ST}}
\nc{\tbart}{\thickspace|\thickspace}
\nc{\Po}{\mathscr{P}}
\nc{\M}{\mathbb{M}}
\nc{\N}{\mathbb{N}}
\nc{\Mod}{\mathscr{M}}
\nc{\llr}{\thickspace\Longleftrightarrow\thickspace}
\renewcommand{\subset}{\subseteq}
\renewcommand{\phi}{\varphi}
\nc{\props}{(p_i)_{i \in I}}
\nc{\Props}{(P_i)_{i \in I}}
\nc{\KM}{\mathcal{K}(\M)}
\nc{\INQML}{\mathrm{InqML}}
\nc{\FO}{\mathrm{FO}}
\nc{\MSO}{\mathrm{MSO}}
\nc{\ML}{\mathrm{ML}}
\nc{\eqinqml}{\equiv_{\INQML}}
\nc{\eqml}{\equiv_{\ML}}
\nc{\eqinqmlbulk}{\equiv^{bulk}_{\INQML}}
\nc{\tupw}{\textbf{w}}
\nc{\tupx}{\textbf{x}}
\nc{\tupy}{\textbf{y}}
\nc{\tuppy}{\textbf{y}_{\psi}}
\nc{\tuppxphi}{\textbf{x}_{\phi}}
\nc{\tuppwphi}{\textbf{w}_{\phi}}
\nc{\supw}{\{\textbf{w}\}}
\nc{\suppxphi}{\{\textbf{x}_{\phi}\}}
\nc{\suppwphi}{\{\textbf{w}_{\phi}\}}
\nc{\hnt}{\rule{0em}{3ex}}
\newenvironment{romanenumerate}
{\begin{list}{(\roman{enumi})}{\usecounter{enumi}
\setlength{\labelwidth}{2cm}
\setlength{\itemindent}{0pt}
\setlength{\itemsep}{0.5\itemsep}
\setlength{\topsep}{\itemsep}
\setlength{\parsep}{0pt}
}}{\end{list}}
\nc{\bre}{\begin{romanenumerate}}
\nc{\ere}{\end{romanenumerate}}
\newenvironment{alphaenumerate}
{\begin{list}{(\alph{enumii})}{\usecounter{enumii}
\setlength{\labelwidth}{2cm}
\setlength{\itemindent}{0pt}
\setlength{\itemsep}{0.5\itemsep}
\setlength{\topsep}{\itemsep}
\setlength{\parsep}{0pt}
}}{\end{list}}
\nc{\bae}{\begin{alphaenumerate}}
\nc{\eae}{\end{alphaenumerate}}

\begin{document}

\title{A First-Order Framework for Inquisitive Modal Logic}
\author{Silke Mei\ss ner and Martin Otto\\
Department of Mathematics\\
Technische Universit\"at Darmstadt}
\date{April~2021}
\maketitle

\begin{abstract}
We present a natural standard translation 
of inquisitive modal logic $\INQML$ into 
first-order logic over the natural two-sorted relational 
representations of the intended models, 
which captures the built-in higher-order features of $\INQML$.
This translation is based on a graded notion of 
flatness that ties the inherent second-order,  
team-semantic features of $\INQML$ over information
states to subsets or tuples of bounded size. 
A natural notion of pseudo-models, which relaxes the
non-elementary constraints on the intended models, 
gives rise to an elementary, purely model-theoretic proof 
of the compactness property for $\INQML$. 
Moreover, we prove a Hennessy-Milner theorem
for $\INQML$, which crucially uses 
$\omega$-saturated pseudo-models and the new
standard translation. 
As corollaries we also obtain van Benthem style characterisation theorems.
\end{abstract}

\section{Introduction}
\label{introsec}

\emph{Inquisitive logics} have recently been expounded systematically
by Ciardelli in~\cite{Ciardelli}, following up on previous work with
Roelofsen~\cite{CiardelliRoelofsen} and earlier sources especially in
the work of Groenendijk cited there. 
The fundamental
motivation is to provide logics with expressive means to deal
not just with assertions but also with questions. While the general 
programme can be carried out systematically for various logics,
like propositional logic in \cite{CiardelliRoelofsen} and first-order
logic in \cite{Ciardelli, CiardelliGrilletti, Grilletti2}, 
it certainly seems particularly
natural also at the level of modal logics, as outlined
in~\cite{Ciardelli},  where Ciardelli gives a first detailed account
of inquisitive modal logic in 2016.  In its epistemic interpretation, for instance, 
modal logic offers the natural classical framework for distinctions between 
different states of affairs (facts, about which basic assertions can be made) 
and cognitive states (information states, about which more complex
assertions, e.g.\ concerning knowledge, can be made). This is a very
natural context in which one may want to give semantics also to
questions. The study of questions is more generally well motivated
-- also at the more foundational philosophical or linguistic level -- 
by considerations about language and logic in all kinds of scenarios that
relate facts, knowledge and information. For instance one may want to
account for the conceptual difference between `\emph{knowing that}'
and `\emph{knowing whether}' something is the case.
And indeed, inquisitive modal logic provides connectives and modal
operators that neatly capture such distinctions; in particular it also
offers, right at the propositional level, a non-trivial disjunction of
alternatives $p$ and $\neg p$, whose semantics captures the idea of
the question `\emph{whether~$p$}', which is suggestively denoted as $?p$.
This novel formula $?p$ is meant to specify, as a kind of  $p/\neg p$ alternative, 
information states that support one of the admissible answers, but
crucially without specifying which one. 
The semantics for inquisitive modal logic is given in terms of
satisfaction of formulae in information states 
(\emph{support semantics} in~\cite{CiardelliRoelofsen,Ciardelli}),
i.e.\ in sets of possible worlds rather than in individual possible worlds. 
This latter feature also accounts for the conceptual 
links between the semantics of inquisitive logics and \emph{team semantics}
for logics of dependence~\cite{Hodges,Vaananen}.
Not too surprisingly maybe,
the semantic modelling for such phenomena in a modal framework 
involves not just possible 
worlds and relations between them (as is the case for basic modal logic) 
but includes information states as primary objects, together with
relationships between (sets of) possible worlds and (sets of) information states. 

In this sense, the setting of inquisitive modal logic puts an 
extra level of (set-theoretic)
complexity on top of the familiar modal modelling 
(cf.~\cite{MLtextbook,GorankoOtto}).
For instance, in the
epistemic setting where possible worlds are associated with 
\emph{states of affairs}:
where Kripke models assign to possible worlds 
sets of alternative possible worlds (\emph{information states}),
inquisitive models assign to possible worlds sets of such 
possible information states
(\emph{inquisitive states}), which may be thought of as possible
answers, or possible information updates. Correspondingly,
natural relational encodings involve both the set of worlds
and a set of information states, as two relevant sorts on an equal
footing. The way these sorts live in a base set and its power set
already suggests a degree of logical complexity that might 
be more of a challenge for a direct first-order account of the semantics of
$\INQML$ than the well-known standard translation for basic modal logic $\ML$.
Indeed, the relational encodings of the intended inquisitive models
form a non-elementary class,
due to non-elementary closure conditions on the universe of information states. 
So, for instance, the straightforward compactness
argument for $\ML$, which just applies first-order compactness through
standard translation, cannot be used directly.

From a model-theoretic perspective, translating a non-classical logic like $\INQML$ into
the classical framework of first-order logic $\FO$ 
provides a systematic advantage. With a standard
translation we can investigate $\INQML$ in the well-known $\FO$ setting and 
explore model-theoretic features of $\INQML$, such as
compactness, in this context.

Compared to previous translations as used
in~\cite{CiardelliOtto,CiardelliOttoLong},  the one we present
here combines several advantages: it puts minimal requirements on the relational
encodings of the inquisitive models; it is more directly defined 
by natural induction on the full syntax of $\INQML$ and is consequently
more intuitive. 
At the technical level, the main novelty is an application of a
concept we call \emph{graded flatness}. The same concept was previously
considered in (first-order) dependence logic in \cite{Kontinen} and in
the context of inquisitive first-order logic in \cite{Grilletti} under
the name of \emph{coherence}, but its application in the modal
context seems to represent an innovation.

At a conceptual level, our treatment may also suggest the 
relaxed class of models, which we call \emph{pseudo-models} and whose
relational counterparts form an elementary class, as an alternative to
the intended models for $\INQML$, whose relational counterparts form a
non-elementary class. Somewhat surprisingly, pseudo-models 
faithfully reflect some of the most salient logical features.
Analogous ideas appear to have been considered 
in neighbourhood semantics~\cite{Pacuit}, 
where monotonic models do not translate into an elementary class, while 
general neighbourhood models do and in some 
connections can be used in the analysis of the 
former~\cite{HansenKupkePacuit}.
But uses of intermediate classes of `weak' models 
have also more classically been explored, e.g.\ 
in the context of monadic second-order logic or of
topological model theory~\cite{Ziegler}.    
We exemplify the value of pseudo-models and our standard translation 
as intermediaries in model-theoretic arguments in two separate
applications. 

The first application is a new, purely model-theoretic argument to 
establish compactness for $\INQML$.\footnote{$\INQML$ compactness is
known from~\cite{Ciardelli}, where it arises as a corollary to a completeness result.}
In the classical case for basic model $\ML$, the standard translation into $\FO$ suffices to 
put $\ML$ into the elementary setting of Kripke structures as 
simple relational structures. The case of $\INQML$ needs to appeal to
pseudo-models because the relational counterparts of the intended
models do not themselves form an elementary class. 

The second application yields a Hennessy-Milner theorem for $\INQML$
w.r.t.\ the natural notion of inquisitive bisimulation studied
in~\cite{CiardelliOtto}. Here the intended models, which support the
Hennessy-Milner property, are obtained from $\omega$-sa\-tu\-ra\-ted
pseudo-models via a 
non-elementary closure operation. Incidentally, $\omega$-saturation 
is also shown to be incompatible with the closure requirements on
actual models, which again highlight the crucial r\^ole of pseudo-models  
as intermediaries. 

Finally we obtain variants of the van~Benthem
style expressive completeness result for $\INQML$
from~\cite{CiardelliOtto} 
(for both, the world- and the state-pointed case) 
as a corollary to the combination
of the compactness and Hennessy-Milner assertions. It is different from
the results in~\cite{CiardelliOtto} in 
referring to a natural strengthening of bisimulation invariance that
bridges the gap between relational models and pseudo-models; it 
is also not applicable in the context of finite model theory.

\medskip
Our choice of terminology is meant to make contact with team
semantic notions to reflect the close relationship between the
team semantic treatment of dependence logics \cite{Vaananen}
and the setting of $\INQML$. Sets of worlds (i.e.~information states) 
can be seen as teams, and in our two-sorted view are treated as 
first-class objects 
(of the second sort) along with worlds (as objects of the first sort).

\section{Inquisitive modal logic}
\label{inqmlsec}

By $(p_i)_{i \in I}$
we denote a set of propositional variables.
Following the terminology in~\cite{CiardelliRoelofsen,Ciardelli}
we associate the following two kinds of \emph{states} 
with a given non-empty set $W$ 
of \emph{possible worlds}.
\\
\emph{Information states}: 
any subset $s\subset W$ is called an \emph{information state};
the set of information states over $W$ is $\Po(W)$.
\\
\emph{Inquisitive states}:
a non-empty set $\Pi$ of information states,  
$\Pi \subset \Po(W)$,
is an \emph{inquisitive state} if it is closed w.r.t.\  set
inclusion: $s \in \Pi$ implies $t \in \Pi$ for all $t \subset s$;
the set of inquisitive states over $W$ is the set of 
those non-empty sets of information states in $\Po(\Po(W))$ that 
satisfy this characteristic \emph{downward closure} condition.
	
\begin{definition}[(inquisitive, modal) models]
\label{inqmm}\mbox{}\\
Let $W$ be a set of possible worlds and 
$\Sigma: W\rightarrow\Po(\Po(W)) \setminus \{
\emptyset\}$ be a function that assigns 
an inquisitive state $\Sigma(w)$ to every world $w\in W$ (an
inquisitive assignment) and  
$V: \props\rightarrow \Po(W)$ 
a function that assigns a subset of $W$ to every propositional variable 
(a propositional assignment). 
Then $\M=(W, \Sigma, V)$ is called an \emph{(inquisitive modal) model}.
\\
With an inquisitive model $\M= (W,\Sigma,V)$
we associate an induced \emph{Kripke model} 
$\KM = (W,\sigma,V)$, where 
$\sigma \colon W \rightarrow \Po(W)$ is defined as 
$\sigma(w):=\bigcup\Sigma(w)$ (a modal assignment).
\end{definition}

A \emph{state-pointed} (inquisitive modal) model is a pair $\M, s$ which
consists of a model $\M$ together with a distinguished information 
state $s\subset W$.
If $s$ is a \emph{singleton information
  state}, i.e.\ $s=\{w\}$ for some $w\in
W$, we also speak of a \emph{world-pointed} inquisitive model
 $\M, \{w\}$ for which we also write just $\M, w$.

Note that the associated Kripke structure reduces the inquisitive
assignment (of inquisitive states $\Sigma(w) \in \Po(\Po(W))$)
to an assignment of single information states $\sigma(w) = \bigcup
\Sigma(w) \in \Po(W)$ that can be cast as sets of successors w.r.t.\ a
modal accessibility relation. The natural relational encoding of
$\sigma$ is in terms of the accessibility relation 
\[
R = \{ (w,w') \colon w' \in \sigma(w) \} \subset W\times W, 
\]
so that 
$\sigma(w)$ becomes the set of immediate successors of $w$
w.r.t.\ $R$, i.e.\ the set $\sigma(w) = R[w] =\{ w' \colon (w,w') \in R \}$.
A corresponding, natural relational encoding of the inquisitive
assignment will have to resort to a two-sorted encoding with a second sort of
information states (from $\Po(W)$) besides the first sort $W$ of
worlds (see~Section~\ref{reprsubsec} below). In this two-sorted scenario, however,
the characteristic downward closure condition on the inquisitive
states $\Sigma(w)$ would remain non-elementary. This motivates the  
following relaxation of the notion of models to what we call 
\emph{pseudo-models}, which
may also be cast in the model-theoretic tradition of approximate or 
weak models that reduce the complexity of higher-order
features, similar to,
e.g.~the use of weak models in topological model theory \cite{Ziegler}.
As we shall see below, this concept can serve 
here as a useful tool for the analysis of the intended, proper models.

Compared to Definition~\ref{inqmm}, the following definition of
pseudo-models just waives the
downward closure requirement on inquisitive assignments.

\begin{definition}[pseudo-models and inquisitive closure]
\label{pseudoinqmm}
\label{inqclosure}\mbox{}\\
A \emph{pseudo-(inquisitive modal) model} is a structure
$\M=(W,\Sigma,V)$ over the set of possible worlds 
$W$ with propositional assignment $V: \props\rightarrow \Po(W)$ 
and a function $\Sigma: W\rightarrow\Po(\Po(W)) \setminus \{
\emptyset\}$, which assigns a non-empty subset 
$\Sigma(w)$ of $\Po(W)$ but not necessarily an inquisitive
state to every world $w \in W$.

With any pseudo-model $\M = (W,\Sigma,V)$ we associate 
its \emph{inquisitive closure} $\M\!\downarrow :=
(W,\Sigma\!\downarrow,V)$, 
which is the proper model whose inquisitive
assignment $\Sigma\!\downarrow $ is induced by $\Sigma$ according to 
\[
\begin{aligned}[t]
\Sigma\!\downarrow: W&\longrightarrow\Po(\Po(W))\text{,}
\\w&\longmapsto\Sigma\!\downarrow\!(w):= 
\{t\in\Po(W)\colon t\subset s\text{ for some }s\in\Sigma(w)\}.
\end{aligned}
\]
\end{definition}

We note that the distinction between a pseudo-model and its
inquisitive closure is immaterial at the level of the associated 
Kripke models with their modal assignment
$\sigma(w) = \bigcup \Sigma(w) = \bigcup \Sigma\!\downarrow\!(w)$.

\begin{definition}[$\INQML$: syntax]\label{syntax}\mbox{}\\
The basic syntax of $\INQML$ is given by the grammar
\[
\phi::=p\tbart\bot\tbart(\phi\wedge\phi)\tbart(\phi\rightarrow\phi)\tbart(\phi\stor\phi)\tbart\Box\phi\thickspace|\boxplus\phi,
\]
with negation, disjunction and diamond $\Diamond$ treated as
abbreviations according to $\neg\phi:=\phi\rightarrow\bot$,
$\phi\vee\psi:=\neg(\neg\phi\wedge\neg\psi)$, and 
$\Diamond\phi:=\neg\Box\neg\phi$.
\end{definition}

In \cite{Ciardelli} the symbol $\store$ is called \emph{intuitionistic disjunction}, $\Box$ the \emph{universal modality} and $\boxplus$ the \emph{inquisitive modality}.

The following extends the standard definition of the semantics of 
$\INQML$ from~\cite{CiardelliRoelofsen,Ciardelli} to
pseudo-models in a straightforward manner.

\begin{definition}[$\INQML$: semantics]\label{semantics}\mbox{}\\
Let $\M=(W,\Sigma,V)$ be a model or a 
pseudo-model, $s\subset W$ an information state. 
The semantics of $\INQML$ is defined as follows.
\begin{itemize}
	\item $\M,s\models p :\llr s\subset V(p)$
	\item $\M,s\models \bot :\llr s=\emptyset$
	\item $\M,s\models \phi\wedge\psi :\llr \M,s\models\phi$ and $\M,s\models \psi$
	\item $\M,s\models \phi\rightarrow\psi :\llr \forall t\subset s: \M,t\models\phi \Rightarrow \M,t\models \psi$
	\item $\M,s\models \phi\stor\psi :\llr \M,s\models\phi$ or $\M,s\models \psi$
	\item $\M,s\models \Box\phi:\llr \forall w\in s: \M,\sigma(w)\models\phi$
	\item $\M,s\models \boxplus\phi:\llr \forall w\in s\thickspace\forall t \in \Sigma(w): \M,t\models\phi$
\end{itemize}
\end{definition}

We note that the semantic clause for 
implication in Definition~\ref{semantics} refers to \emph{all} subsets $t
\subset s$, over models or pseudo-models alike. Similarly,
the downward closure condition known as \emph{persistency}, as
discussed in the following observation,  speaks about 
\emph{all} subsets $t \subset s$ of the given information state $s$,
also when interpreted in the non-standard setting of pseudo-models.

\begin{observation}
\label{props}
Over all models as well as 
pseudo-models $\M$, $\INQML$ has the following properties, for all 
$\phi \in  \INQML$:
\bre
\item
$\M,s\models\phi$ implies 
$\M,t\models\phi$ for all $t\subset s$;      
\item
$\M,\emptyset\models\phi$.
\ere
\end{observation}

Property~(i) is called 
\emph{persistency} in~\cite{CiardelliRoelofsen,Ciardelli} 
(and usually referred to as \emph{downward closure}
in team semantic terminology), while~(ii) is called  
\emph{semantic ex-falso} (reflecting the \emph{empty team property}).

The following gives a further indication that
extension of the semantics of $\INQML$ beyond the intended
inquisitive models is very natural.
The proof is by straightforward 
syntactic induction following the clauses of
Definition~\ref{semantics}; for the $\boxplus$-case one uses persistency, and for the $\Box$-case, which refers to the
associated Kripke structure, one uses the fact that the associated modal assignment $\sigma$
is the same for the pseudo-model and its inquisitive closure.

\begin{proposition}\label{pseudoclosuresame}
Let $\M$ be a pseudo-model, $\M\!\downarrow$ its
inquisitive closure, and $s \subset W$ any information state over
their common universe $W$ of possible worlds. Then 
for $\phi\in \INQML$ we have
\[
\M\!\downarrow,s\models\phi
\;\llr\;
\M,s\models\phi.\]
\end{proposition}

This indicates that, as far as e.g.\ 
deductive reasoning is concerned, $\INQML$ might as well be cast
in the extended setting of pseudo-models. The difference is important, 
though, e.g.\ in issues concerning what is the natural habitat for the
key notion of model equivalence, viz.\ inquisitive bisimulation 
equivalence~\cite{CiardelliOtto} to be discussed in Section~\ref {HMsec}, 
or how $\INQML$ embeds into classical logics of reference like~$\FO$
(cf.\ key results in \cite{CiardelliOtto} vs.\ Corollary~\ref {vBcor} here).

\section{Standard translation}

A standard translation serves as a semantically adequate (one
could say, truthful) translation between logical frameworks.

Recall the well-known standard translation of modal logic $\ML$ into first-order
logic $\FO$. It is based on the straightforward transcription of the
clauses for the Kripke semantics of $\ML$ into their natural first-order 
analogues
over Kripke models viewed as ordinary relational
first-order structures. 
The situation for $\INQML$ is different, 
because inquisitive (pseudo-)models are naturally rendered as two-sorted
rather than ordinary single-sorted relational structures.\footnote{This
aspect of two-sortedness is similar to the treatment of neighbourhood
models~\cite{Pacuit}, but $\INQML$ modelling imposes a different and
in some sense tighter link between sorts.} 
Since an inquisitive assignment is a function
from the set $W$ of possible worlds to sets of sets of possible worlds, it
is of inherently higher type than a modal assignment. Its natural
relational encoding consists of a binary relation not over $W$ 
itself, but between $W$  (as a first sort) and a set $S$ of 
information states (as a second sort) where $S \subset \Po(W)$. 
In the following we discuss a setting and format for a standard
translation of $\INQML$ into $\FO$ in the natural two-sorted 
relational framework that is similar in spirit to that in~\cite{CiardelliOtto}
but more liberal and more uniform
(cf.~Remark~\ref{compareSTrem}).
The technical novelty underpinning
this new approach is the application of a graded notion of flatness, 
\emph{graded flatness}, to $\INQML$, as independently developed
by the first author in~\cite{Meissner}.

We want to associate the semantics of inquisitive modal
logic $\INQML$ over models (or even pseudo-models) with the
semantics of first-order logic over associated two-sorted 
relational (pseudo-)models. As usual, this task involves two
translation levels that need to go hand in hand: transformations 
linking the underlying (pseudo-)models $\M$ to
relational representations $\Mod$, from which the underlying 
(pseudo-)models $\M$ can be recovered as $\M = \M(\Mod)$; 
and a translation of formulae 
$\phi \in \INQML$ into formulae $\phi^\ast \in \FO$ such that 
\[
\M(\Mod),s \models \phi \llr \Mod,s \models \phi^\ast.
\]

\subsection{Relational representations of models}
\label{reprsubsec}

As relational counterparts of inquisitive (pseudo-)models we 
consider two-sorted relational structures of the form
\[
\Mod=(W,S,\epsilon,E,\Props)
\]
with some non-empty sets $W$ 
and $S$ as first and second sorts, linked by 
two mixed-sorted binary relations 
$\epsilon,E\subset W\times S$, and with
unary predicates $P_i\subset W$ over the first sort for all 
$p_i$, $i \in I$, that encode the propositional assignment as usual.
The intended r\^oles of $\epsilon$ and $E$ are as follows:
$\epsilon \subset W\times S$ encodes membership of possible worlds in information
states, so that $s \in S$ can be associated with 
$\underline{s}:=\{w\in W\colon (w,s)\in \epsilon\} \in \Po(W)$; 
and $E \subset W\times S$ encodes the inquisitive assignment as a
relation that associates with each world $w \in W$ the
set $E[w]:=\{s\in S\colon (w,s)\in E\}
\in \Po(S)$.

\begin{definition}[relational (pseudo-)models]\label{relinqmm}
\label{relpseudoinqmm}
\mbox{}\\
A structure $\Mod$ of the type above 
is a \emph{relational (inquisitive modal) model} if the following
conditions are satisfied
for all $w \in W$, $s,t \in S$ and $a \subset S$:
\bre
	\item $\underline{s}=\underline{t} \;\Rightarrow\; s=t$
          (extensionality);
	\item $E[w]\not=\emptyset$ (non-emptiness);
	\item if $a\subset \underline{s}$ for $s\in E[w]$, then $a = \underline{t}$
          for some $t\in E[w]$ (downward closure).
\ere
Correspondingly, a structure $\Mod=(W,S,\epsilon,E,\Props)$ of the same
format is a \emph{relational (inquisitive modal) pseudo-model} if it
satisfies (i) and (ii).
\end{definition}

Relational (pseudo-)models with distinguished information 
states $s \in S$ are described as 
\emph{state-pointed} or, in the case of singleton states $\underline{s} = \{ w
\}$ as \emph{world-pointed}, in analogy with the terminology for models.  
Due to extensionality~(i), we shall identify information states $s \in S$ 
with sets of worlds $\underline{s} \subset W$ and regard the second
sort $S$ as a subset of the power set $\Po(W)$, with $\epsilon$ as
the actual membership relation between $W$ and $S \subset \Po(W)$.

\begin{observation}\label{pseudoelementary}
The class $\mathscr{C}$ of all state-pointed (respectively
world-pointed) relational pseudo-models is $\Delta$-elementary, i.e.\
there exists a set of formulae $\Phi\subset\FO$ such that
$\mathscr{C}=$Mod$(\Phi)$.\footnote{If the set of propositions is finite, then $\mathscr{C}$ is 
even elementary, i.e.\ definable by a single $\FO$-formula.}
\end{observation}

It is fairly easy to see that the class of all relational models cannot be 
$\Delta$-elementary, as downward closure (condition~(iii) in
Definition~\ref{relinqmm}) cannot be expressed without reference to arbitrary 
subsets of the first sort. Indeed, if it was $\Delta$-elementary then
so would be the class of all \emph{full relational models}, defined by the
additional condition that $S = \Po(W)$ be the full power set. That,
however, is ruled out by the observation that $\FO$ does not satisfy
compactness over this class of all full relational models: over that class,
$\FO$ captures the full power of monadic second-order logic $\MSO$ over the
first sort; so it can, e.g., define the class corresponding to Kripke models
that satisfy the well-foundedness condition of L\"ob frames (cf.~\cite{CiardelliOtto}).

\begin{observation}
The class of all relational models is not $\Delta$-elementary.
\end{observation}

\paragraph*{From relational (pseudo-)models to (pseudo-)models.}
 \mbox{}\\ 
With any relational (pseudo-)model $\Mod=(W,S,\epsilon,E,(P_i)_{i \in
  I})$ we associate the (pseudo-)model $\M(\Mod) = (W,\Sigma,V)$ 
that decodes the relational information in $\Mod$ into
functional assignments according to
\[
\begin{array}{r@{\;\longmapsto\;}l}
\Sigma \colon  w & \{ s \in S \colon (w,s) \in E \},
\\
\hnt
V \colon p_i & \{ 
\displaystyle w \in W \colon w \in P_i \}.
\end{array}
\]

We observe that the actual extension of the second sort $S \subset \Po(W)$ in
$\Mod$ is immaterial in as far as it may go beyond the range of $\Sigma$.

\paragraph*{From (pseudo-)models to relational (pseudo-)models.}
\mbox{}\\
With a (pseudo-)model $\M =(W,\Sigma,V)$ and a distinguished
information state $s \subset W$, we associate as a relational representation 
any relational (pseudo-)model 
$\Mod=(W,S,\epsilon,E,(P_i)_{i \in I})$ that encodes $\Sigma$ and $V$ 
over sorts $W$ and $S$, where $S \subset \Po(W)$ is rich enough to 
represent the image $\Sigma(w) \subset \Po(W)$ for all $w \in W$
as well as the distinguished information state $s$:
\[
\begin{array}{r@{\;}c@{\;}l}
S &\,\subset& 
\displaystyle 
\Po(W) \mbox{ with } S \supseteq \{ s \} \;\cup \bigcup_{w \in W} \!\Sigma(w),
\\
\hnt
\epsilon &:=& \{ (w,s) \in W \times S \colon w \in s \} = \; \in\restriction\!
  (W\!\times\! S),
\\
\hnt
E &:=& \{(w,s)\in W\times S \colon s\in\Sigma(w) \},
\\
\hnt
P_i &:=& \{ w\in W\colon w\in V(p_i) \} \mbox{ for } i \in I.
\end{array}
\]

Note that $\Mod$ is fully determined by $\M$ once the 
actual extension of the second sort $S$ is fixed; that however is 
naturally only subject to a richness condition. (We argued above that
insistence on fullness, i.e.\ on the maximal extension $S = \Po(W)$,
may not be advisable.)

\begin{definition}[relational representations]
\label{transmodel}
\mbox{}\\
A relational (pseudo-)model $\Mod$ is a \emph{relational representation}
of a given (pseudo-)model $\M$ precisely for $\M = \M(\Mod)$.
\\
A state-pointed relational (pseudo-)model $\Mod,s$
is a \emph{relational representation} of the state-pointed (pseudo-)model $\M,s$
if in addition the distinguished information state $s$ is represented as an
element of its second sort $S$.\footnote{
Note that the distinguished state becomes an element of the second sort 
of~$\Mod$, hence available as a parameter-definable subset of the first sort, but not 
as a constant or predicate.}
\end{definition}

It is clear from the above that every state-pointed (pseudo-)model
admits relational representations, and that every relational
(pseudo-)model $\Mod$ represents a unique (pseudo-)model, viz.\ 
$\M(\Mod)$.

\subsection{Graded flatness and the standard translation}
\label{gfandstsubsec}

Compared to the well-known standard translation for plain modal 
logic over Kripke models 
(for which $\M$ and $\Mod$ are practically identical), $\INQML$ 
involves challenges associated with the semantics of implication and $\Box$.
The corresponding clauses in Definition~\ref{semantics} involve reference
to information states that might not necessarily be directly available 
in the second sort of $\Mod$. 

There are different suggestions to overcome this problem.
An elimination of $\Box$ is possible via so-called
\emph{resolutions} \cite{Ciardelli} or via the $\Box$-free
characteristic formulae for finitary bisimulation classes
from~\cite{CiardelliOtto};
then a straightforward translation can be given. That 
standard translation, however, is based on stronger closure conditions
on the universe of information states in the relational
encodings of the inquisitive models, which give $\FO$ direct access to 
the relevant information states.
Such stronger closure conditions may further interfere with compactness 
over the required classes of models, as mentioned in \S~\ref{reprsubsec}.
A straightforward translation without elimination of $\Box$ is also 
possible
but also requires stronger closure conditions on the class of relational models.

\begin{remark}
\label{compareSTrem}
Compared to the previous translations, the present proposal is more
general, more uniform and more direct:
it works for the natural class of all 
pseudo-models and does not require any additional (elementary or non-elementary) 
closure conditions on the state universe of the intended relational models.
Moreover, it is defined directly by induction on the 
unrestricted standard syntax of $\INQML$, without appeal to specific 
syntactic normal forms.
\end{remark}

Our standard translation relies on the following concept of \emph{graded flatness}, which
had also been investigated in the team-semantic context of dependence logic
(cf.~\cite{Kontinen}) and in the context of inquisitive first-order
logic (cf.~\cite{Grilletti}) under the name of \emph{coherence}.
For our context it
was independently (re-)discovered and put to this new use in~\cite{Meissner}.
Our preferred terminology of \emph{graded flatness}
derives from the notion of \emph{flatness} in team semantics. If we
think of information states (sets of worlds) as teams, then 
a formula $\phi \in \INQML$
would be \emph{flat} (in the team semantic sense) if its truth 
in $s$ is equivalent to truth in $\{w\}$ for all $w \in s$.
Graded flatness generalises this idea to quantitative bounds on the
size of subsets $s' \subset s$ that need to be investigated, 
rather than singleton subsets. Such a size bound can be
obtained as a syntactic parameter as follows.

\begin{definition}[flatness grade]\label{flatnessgrade}
\mbox{}\\
The \emph{flatness grade $\flat (\phi)\in\N$} of $\phi \in \INQML$ is 
defined by syntactic induction, for all $\psi,\chi \in \INQML$, according to
\begin{itemize}
	\item[--] $\flat(\phi) := 0$ for 
atomic $\phi$ and all $\phi$ of the form $\Box\psi$ or $\boxplus\psi$;
	\item[--] $\flat (\psi\wedge\chi) := \max\{\flat (\psi), \flat(\chi)\}$;
	\item[--] $\flat (\psi\rightarrow\chi) := \flat(\chi)$;
	\item[--] $\flat (\psi\stor\chi) := \flat(\psi) + \flat(\chi) + 1$.
\end{itemize}
\end{definition}

\begin{proposition}[graded flatness]\label{gradedflatness}
\mbox{}\\
Inquisitive modal logic $\INQML$ satisfies the following \emph{graded flatness
  property}.
	For all $\phi\in \INQML$
and state-pointed (pseudo-)models $\M,s$:
\[\M,s\models\phi \;\llr\; \M,t\models\phi
\mbox{ for all }\, t \subset s\thickspace \mbox{ of size } 
\left|t\right|\leq\flat (\phi)+1.
\]
\end{proposition}

\begin{proof}
The direction from left to right follows immediately from persistency
for $\INQML$. The implication from right to left 
is shown by syntactic induction on $\phi$. 
We illustrate the $\store$-case, which is the most interesting.
\\
For $\phi=\psi\stor\chi$  let 
$m:= \flat(\psi)$ and $n :=\flat(\chi)$ so that
$\flat(\phi)=m+n+1$, and assume that
$\M,t\models\psi\stor\chi$
for all $t \subset s$ of size $\left|t\right|\leq m+n+2$.
By \Cref{semantics} and the induction hypothesis we get that
for all $t \subset s$ with $\left|t\right|\leq m+n+2$:
\[
\begin{array}{rrl}
&
\forall a \subset t \text{ with }\left|a\right|\leq
m+1: &\M,a\models\psi
\\
\mbox{ or } &
\forall a \subset t \text{ with }\left|a\right|\leq
n+1: &\M,a\models\chi.
\end{array}
\]

It follows that
\[
\begin{array}{rrl}
&
\forall t \subset s\text{ with }\left|t\right|\leq m+1: &
\M,t\models\psi
\\
\mbox{ or }
& \forall t \subset s\text{ with }\left|t\right|\leq n+1:
& \M,t\models\chi.
\end{array}
\]
Indeed, if this were false, there would exist information 
states $t_{1}\subset s$ with $\left|t_{1}\right|\leq m+1$ 
and $t_{2}\subset s$ with $\left|t_{2}\right|\leq n+1$ 
such that $\M,t_{1}\not\models\psi$ and $\M,t_{2}\not\models\chi$. 
But then, for $t_{0}:=t_{1}\cup t_{2}$, the previous statement 
would be false: we have $\left|t_{0}\right|\leq m+n+2$ 
but $t_{1}\subset t_{0}$ violates the first disjunct and 
$t_{2}\subset t_{0}$ violates the second disjunct. 

With the induction hypothesis we get $\M,s\models\psi\stor\chi$.
\end{proof}

We use the following notation.
Generally, we 
take variable symbols 
$x,y, \ldots$ to be interpreted 
over the first sort (worlds), and variable symbols
$\lambda,\mu,\ldots$ over the second sort (information states).
A tuple of length $n$ is denoted as $\tupx=(x_1,...,x_n)$, 
the set of its components as $\{\tupx\} = \{ x_1,...,x_n\}$. 
If the length of a tuple is determined by some flatness grade 
$\flat(\phi)$, we write $\tuppxphi$ for the tuple
$(x_1,...,x_{\flat(\phi)+1})$ and $\suppxphi$ for the associated set.
For better readability we write $x\!\in\!\lambda$ instead of
$\epsilon x\lambda$ in the following first-order formulae.

Our standard translation $\phi \mapsto \phi^\ast(\lambda)$ 
is defined below, by syntactic induction on $\phi \in \INQML$. 
The interesting, somewhat non-standard feature involves the 
necessary passage between sorts: on one hand, $\phi$ is translated
into the first-order formula $\phi^\ast(\lambda)$ in a free variable
$\lambda$ of the second sort,
which is to be interpreted as the distinguished information
state $s$; the core induction, on the other hand, deals with auxiliary 
formulae $\ST(\xi,\tupx)$
in tuples $\tupx$ of free variables of the 
first sort that capture the semantics of $\xi^\ast(\mu)$
for $\mu = \{\tupx\}$.

\begin{definition}[standard translation]\label{standardtranslation}
\mbox{}\\
For $\phi\in \INQML$ define its \emph{standard translation} 
$\phi^*(\lambda)\in \FO$ in one free state variable
$\lambda$ as
\[
\phi^*(\lambda):= \forall
\tuppxphi
\Bigl(\!\!\!
\bigwedge_{\quad k \leq\flat(\phi)+1} \!\!\!\!\!\!\!\!\!
x_k \!\in\! \lambda \;\; \longrightarrow \; \ST(\phi,\tuppxphi)\Bigr),
\]
where the auxiliary first-order formulae $\ST(\xi,\tupx)$,
with free world
variables among $\tupx$, are defined by syntactic induction according to:
\begin{itemize}
\item[--] 
$\displaystyle  \ST(p,\tupx):=\underset{k\leq n}{\bigwedge}Px_k$
\item[--]
$\displaystyle \ST(\bot,\tupx):= \bigwedge_{k\leq n} \neg x_k \!=\! x_k$
\item[--]
$\displaystyle \ST(\psi\wedge\chi,\tupx):=\ST(\psi,\tupx)\wedge \ST(\chi,\tupx)$
\item[--]
$\ST(\psi\rightarrow\chi,\tupx):=\forall\tupy \Bigl[ \Bigl(
\bigwedge_{k\leq n} \bigvee_{l\leq n}  
y_k=x_l\Bigr) \longrightarrow\bigl( \ST(\psi,\tupy)\rightarrow
\ST(\chi,\tupy)\bigr)\Bigr]$
\item[--]
$\displaystyle \ST(\psi\stor\chi,\tupx):=\ST(\psi,\tupx)\vee \ST(\chi,\tupx)$
\item[--] 
$\displaystyle \ST(\Box\psi,\tupx):=\underset{k\leq n}{\bigwedge}
\forall \tuppy \forall \boldsymbol{\mu}_{\psi}
\Bigl( \!\!\!\!
\bigwedge_{\quad l \leq \flat(\psi)+1}
\!\!\!\!\!\!\!\!
\bigl(
Ex_k \mu_l\wedge y_l \!\in\! \mu_l
\bigr)
\; \longrightarrow \; \ST(\psi,\tuppy) \;\Bigr) $
\item[--] 
$\displaystyle \ST(\boxplus\psi,\tupx):=\underset{k\leq n}{\bigwedge}\forall \mu
\left(Ex_k \mu\rightarrow\psi^*(\mu)\right)$
\end{itemize}
\end{definition}

\noindent
The intuition for, e.g.\ the translation 
of $\Box$ in this definition  is the following.
We want to suitably mimic the semantics of $\Box$ in our context, 
hence we want to say that
for all worlds $x_k\in \tupx$, $\psi^*(\lambda)$ is satisfied by the
information state $\sigma(x_k)$.
For this we need access to the states $\sigma(x_k)$.
Via the subformula $Ex_k \mu_l\wedge y_l \in \mu_l$ of our translation we can check whether a
world $y_l$ is element of $\sigma(x_k)$.
Since $\sigma(x_k)$ is in general not necessarily represented as an
element of the second sort
of the relational pseudo-model that is supposed
to satisfy the translation, we cannot access it directly via a state variable.
Instead we can fix all its substates of size $\flat(\psi)+1$ in the shape of
tuples of worlds via
\[
\forall \tuppy \forall \boldsymbol{\mu}_{\psi}
\Bigl( \!\!\!\! \bigwedge_{\quad l \leq \flat(\psi)+1} \!\!\!\!\!
\bigl( Ex_k \mu_l\wedge y_l \in \mu_l\bigr) \longrightarrow 
\; \ST(\psi,\tuppy)
\Bigr).\]
By \Cref{gradedflatness} this correctly expresses that
$\sigma(x_k)$ satisfies $\psi^*(\lambda)$.

\begin{remark}\label{worldstandardtranslation}
  We may analogously define a standard translation
  $\pi^\ast(x)$ in the single free world variable~$x$
of the first sort, which covers the world-pointed case
just as the standard translation $\phi^\ast(\lambda)$
covers the state-pointed case. 
For this we may just use
$\phi^\ast(x) := \ST(\phi,x)$.\footnote{This naturally also captures
semantics of $\phi^\ast(\lambda)$ for $\lambda = \{ x \}$, as it should.}
\end{remark}

The following shows that the proposed standard translation is adequate --
in preserving the semantics and turning $\INQML$ into a 
\emph{syntactic fragment} of $\FO$ -- over the rich class of relational
encodings of (pseudo-)models.

\begin{proposition}[$\INQML$ as a fragment of $\FO$]
\label{fragment}
\mbox{}\\
Let $\phi\in \INQML$ and let $\phi^*(\lambda)\in \FO$ be its standard translation. Let $\M,s$ be a (pseudo-)model and $\Mod,s$ be a relational representation of $\M,s$. Then
\[\M,s\models\phi\llr\Mod,s\models\phi^*(\lambda),\]
where the state variable $\lambda$ is interpreted as $s$.
\end{proposition}

\begin{proof}
We show below that, for $\phi \in \INQML$,
\[\M,\supw\models\phi\llr\Mod,\tupw \models \ST(\phi,\tupx)\tag{$\ast$}\]
for all finite tuples of worlds $\tupw$ from $W$
and matching tuples $\tupx$ of variables.
From this we obtain the claim of the proposition, that 
$\M,s\models\phi\Leftrightarrow\Mod,s\models\phi^*(\lambda)$,
as follows.
By \Cref{gradedflatness},
$\M,s\models\phi$ is equivalent to
\[
\forall t\subset s \text{ with }
\left|t\right|\leq\flat(\phi)+1:\M,t\models\phi,
\]
which is further equivalent to
\[
\forall \tuppwphi\in s:\M,\suppwphi\models\phi\text{.}
\]
(Note that for the direction from right to left we need
$\M,\emptyset\models\phi$ on the right-hand side,
but the set of all tuples $\tuppwphi\in s$ of size
$1\leq\left|\tuppwphi\right|\leq\flat(\phi)+1$
does not contain the
empty tuple. This, however, is unproblematic for the equivalence 
since $\M,\emptyset\models\phi$ trivially holds by \Cref{props}.)
\\
With $(\ast)$ we find that $\M,s\models\phi$ is equivalent to
\[
\forall\tuppwphi\in s:\Mod,\tuppwphi\models
\ST(\phi,\tuppxphi),
\]
which translates equivalently into
\[
\Mod, s \models\forall \tuppxphi
\Bigl( \!\!\!\!\!\bigwedge_{\quad k \leq \flat(\phi)+1} \!\!\!\!\!\!\!  x_k \!\in\!
\lambda \;\; \longrightarrow \; \ST(\phi,\tuppxphi) \Bigr),
\]
which is the same as $\Mod,s\models\phi^*(\lambda)$, 
by~\Cref{standardtranslation}.

\medskip
It remains to show $(\ast)$ by syntactic induction.
We explicitly treat the $\Box$- and $\boxplus$-steps, 
and show the implication from left to right in each case.

For the case of $\phi=\Box\psi$ assume $\M,\supw\models\Box\psi$,
which, by \Cref{semantics}
and~\Cref{gradedflatness},  
means that
\[
\forall w\in \supw\thickspace\forall t\subset\sigma(w)\text{ with
}\left|t\right|\leq\flat(\psi)+1:\M,t\models\psi.
\]
This is further equivalent to
\[
\forall
w\in\supw\thickspace\forall
\textbf{u}_{\psi}
\in\sigma(w):\M,\{\textbf{u}_{\psi}\}\models\psi,
\]
and by induction hypothesis to
\[
\forall w\in\supw\thickspace\forall \textbf{u}_{\psi}\in\sigma(w):
\Mod,\textbf{u}_{\psi}\models \ST(\psi,\tuppy).
\]
This condition is correctly rendered 
in $\FO$ as 
\[
\displaystyle 
\Mod,\tupw\models 
\ST(\Box\psi,\tupx) =\underset{k\leq n}{\bigwedge}
\forall \tuppy \forall \boldsymbol{\mu}_{\psi}
\Bigl( \!\!\!\!
\bigwedge_{\quad l \leq \flat(\psi)+1}
\!\!\!\!\!\!\!\!
\bigl(
Ex_k \mu_l\wedge y_l \!\in\! \mu_l
\bigr)
\; \longrightarrow \; \ST(\psi,\tuppy) \;\Bigr):
\]
universal quantification over $w\in\supw$ is expressed by the conjunction
over the
$x_k\in\tupx$ enumerating the
$w_k\in\tupw$; the quantification
$\forall \textbf{u}_{\psi}\in\sigma(w)$ is represented 
by the quantification $\forall \tuppy \forall
\boldsymbol{\mu}_{\psi} $, relativised so that the
$y_l$ are instantiated by worlds from $\sigma(w_k)$.

For the case $\phi=\boxplus\psi$ assume $\M,\supw\models\boxplus\psi$.
With \Cref{semantics} and \Cref{gradedflatness} we get
\[\forall w\in \supw\thickspace\forall t\in\Sigma(w)\thickspace\forall a\subset t
\text{ with }\left|a\right|\leq\flat(\psi)+1:\M,a\models\psi,\]
hence
\[\forall w\in \supw\thickspace\forall t\in\Sigma(w)\thickspace\forall \textbf{u}_{\psi}\in t :
\M,\{\textbf{u}_{\psi}\}\models\psi.\]
By induction hypothesis we get 
\[\forall w\in \supw\thickspace\forall t\in\Sigma(w)\thickspace\forall \textbf{u}_{\psi}\in t :
\Mod,\textbf{u}_{\psi}\models \ST(\psi,\tuppy).\]
Again we can express this in FO according to
\[\Mod,\tupw\models\underset{k\leq n}{\bigwedge}\forall \mu
\Bigl[ Ex_k \mu\longrightarrow
	\Bigl(\forall \tuppy\Bigl(
\!\!\!\!
\bigwedge_{\quad l \leq \flat(\psi)+1}
\!\!\!\!\!\!\!\!\! y_l\in \mu
\; \longrightarrow \; \ST(\psi,\tuppy)\Bigr)\Bigr)\Bigr],\]
which is $\Mod,\tupw\models \ST(\boxplus\psi,\tupx)$ by~\Cref{standardtranslation}.
\end{proof}

\begin{remark}\label{worldfragment}
The analogue of \Cref{fragment} for world-pointed models
and our world-version of the standard translation
from~\Cref{worldstandardtranslation}
is easily checked.
\end{remark}

\section{Compactness for $\INQML$}
\label{compactsec}

It is known that $\INQML$ has a sound and strongly 
complete proof calculus and therefore satisfies compactness (cf.~\cite{Ciardelli}). 
We use our standard translation to give a new, purely model-theoretic
proof, essentially by reduction to first-order compactness.
But while the corresponding reduction
is totally straightforward for basic modal logic $\ML$ over Kripke structures,
we here need to deal with the additional complication that the class
of 
relational models is not $\Delta$-elementary. Correspondingly, a
detour through pseudo-models plays an essential r\^ole in our proof.
Moreover, the proof also shows that compactness over the class of all 
pseudo-models works in a straightforward manner. 
We interpret this as an additional, natural indication that $\INQML$
could also be explored over the extended class of all pseudo-models.

We consider the satisfiability version of compactness. Of course, by
semantic ex-falso (cf.\ Observation~\ref{props}), any set 
$\Phi\subset\INQML$ is trivially satisfied by any model
$\M,\emptyset$. 
For a non-trivial statement we need to exclude the empty state.

\begin{proposition}[compactness]\label{compactnessinqml}
\mbox{}\\
$\INQML$ satisfies compactness, i.e.\ a set of formulae
$\Phi\subset \INQML$ is satisfiable by 
some non-empty state of some (pseudo-)model
if, and only if, every finite 
subset of $\Phi$ is satisfiable by 
some non-empty state of some (pseudo-)model.
\end{proposition}

\begin{proof}
Let $\Phi\subset \INQML$ be a set of formulae such that any finite subset
$\Phi_0\subset\Phi$ is satisfiable by a state-pointed model $\M_{\Phi_0},s_{\Phi_0}$
such that $s_{\Phi_0}\not=\emptyset$. We want to show that $\Phi$ is satisfiable as well.
\\Let $\Phi^*(\lambda)\subset \FO$ (respectively $\Phi_0^*(\lambda)\subset \FO$) be the
set of all standard translated formulae of $\Phi$ (respectively $\Phi_0$) and let for
each $\Phi_0\subset\Phi$ the relational model $\Mod_{\Phi_0},s_{\Phi_0}$ be a relational
representation of $\M_{\Phi_0},s_{\Phi_0}$. Then \Cref{fragment} yields
$\Mod_{\Phi_0},s_{\Phi_0}\models\Phi_0^*$ for all finite $\Phi_0^*\subset\Phi^*$.
\\We let $\Delta=\Delta(\lambda)\subset \FO$ be a set of formulae
defining the class of all state-pointed relational pseudo-models 
(see \Cref{pseudoelementary}) 
with non-empty state ($\lambda \not= \emptyset$)
and let $\tilde{\Phi}:=\Phi^*\cup\Delta$. Since
$\Mod_{\Phi_0},s_{\Phi_0}\models\Delta$ for all finite $\Phi_0\subset\Phi$, any finite
subset of $\tilde{\Phi}$ is satisfiable. Hence by compactness of $\FO$, $\tilde{\Phi}$
is satisfiable by some relational pseudo-model $\Mod,s$.
\\Then \Cref{fragment} yields $\M,s\models\Phi$ for $\M,s:=\M(\Mod),s$ and
\Cref{pseudoclosuresame} entails $\M\!\downarrow,s\models\Phi$ for the inquisitive
closure $\M\!\downarrow,s$ of $\M, s$.
\end{proof}

With persistency it is easy to see that compactness over the class of all state-pointed
(pseudo-)models implies compactness over the class of all world-pointed (pseudo-)models.

\section{A Hennessy-Milner class for $\INQML$}
\label{HMsec}

The notion of inquisitive bisimulation 
from~\cite{CiardelliOtto} plays an essential r\^ole 
for~$\INQML$ as does ordinary bisimulation for plain modal logic~$\ML$
(for background cf.~\cite{GorankoOtto}).
In particular, inquisitive $n$-bisimulation equivalence $\sim_n$
between two inquisitive models represents a 
finite approximation to full inquisitive
bisimulation equivalence; at level~$n \in \N$, $\sim_n$ 
is related to $\INQML$-equivalence 
up to nesting depth~$n$ (of the modalities $\Box$ and $\boxplus$),
in the familiar style of an
Ehrenfeucht--Fra\"\i ss\'e\ correspondence for
$\INQML$~\cite{CiardelliOtto} (see Theorem~\ref{EFthm} below).
Full inquisitive bisimulation equivalence $\sim$ between 
world- or state-pointed inquisitive models is naturally defined in
terms of back\&forth systems, bisimulation relations, or winning
strategies for the defender (player~$\mathbf{II}$) in the 
corresponding infinite back\&forth game. The finite levels $\sim_n$
for $n \in \N$ are best understood as approximations to $\sim$ 
in terms of winning strategies for player~$\mathbf{II}$ 
in the $n$-round back\&forth game; 
the common refinement $\sim_\omega$ of all the finite 
levels $\sim_n$ is defined in terms of winning strategies 
for any finite number of rounds, which in general is
characteristically weaker than full $\sim$. The crucial feature of the
back\&forth games for inquisitive bisimulation is that each round
comprises two phases so that the probing of state-pointed 
positions is interleaved with a probing of intermediate world-pointed 
positions (for details of the game and a 
fuller discussion we refer to~\cite{CiardelliOtto}).

The following summarises the Ehrenfeucht--Fra\"\i ss\'e\
correspondence derived in~\cite{CiardelliOtto}. We write 
$\equiv_\INQML^n$ for $\INQML$-equivalence 
between world- or state-pointed inquisitive models
up to nesting depth~$n$, and use the term \emph{$\sim_n$-types} 
for $\sim_n$-equivalence classes (of world- or state-pointed)
inquisitive modal models. 

\begin{theorem}[inquisitive Ehrenfeucht--Fra\"\i ss\'e\ \cite{CiardelliOtto}]
\label{EFthm}
For world- or state-pointed inquisitive modal models over a 
finite signature and $n \in \N$:
\[
\begin{array}{rcl}
\M,w \sim_n \M',w' &\Leftrightarrow &
\M,w \equiv_\INQML^n \M',w', 
\\
\hnt
\M,s \sim_n \M',s' &\Leftrightarrow &
\M,s \equiv_\INQML^n\M',s'. 
\end{array}
\]
In particular, $\INQML$ is preserved under $\sim$, and $\sim_\omega$
coincides with $\INQML$-equivalence over finite signatures.
Moreover, for world- or state-pointed inquisitive modal models $\M,w$
or $\M,s$ over a finite signature, there are \emph{characteristic formulae}
$\chi_{\M,w}^n $and $\chi_{\M,s}^n$ of $\INQML$ that define $\sim_n$-types
(up to persistency, cf.\ Observation~\ref{props}) in the sense that
\[
\begin{array}{rcl}
\M',w' \models \chi_{\M,w}^n&\Leftrightarrow &
\M',w' \sim_n \M,w, 
\\
\hnt
\M',s' \models \chi_{\M,s}^n &\Leftrightarrow &
\M',s' \sim_n \M,s_0 \mbox{ for some } s_0 \subset s.
\end{array}
\]
\end{theorem}

We remark that the characteristic formulae $\chi_{\M,w}^n$, for $\sim_n$-types
of worlds, are \emph{truth-conditional} (in the terminology
of~\cite{Ciardelli,CiardelliOtto}) or \emph{flat} (in the terminology of team
semantics). This means that 
\[
\begin{array}{r@{\;\;\;\Leftrightarrow\;\;\;}l}
\M',s' \models \chi_{\M,w}^n
&
\M',\{ w' \} \models \chi_{\M,w}^n
\mbox{  for all } w' \in s'
\\
\hnt
&
\M',w' \models \chi_{\M,w}^n \;
\mbox{  for all } w' \in s'.
\end{array}
\]

For plain modal logic $\ML$ and many of its relatives, the significant gap
between $\sim$ and $\sim_\omega$, and between $\sim_\omega$ and $\equiv_\ML$
for infinite signatures, is bridged by modally saturated and in particular
by $\omega$-saturated models. This phenomenon, which for basic modal logic
is (a generalisation of) the well-known Hennessy--Milner theorem, is a crucial
tool in the model theory of modal logic~\cite{GorankoOtto}: $\omega$-saturated
Kripke structures are modally saturated and satisfy the Hennessy--Milner
correspondence
\[
\mathcal{K},w  \sim \mathcal{K}',w' 
\; \Leftrightarrow \;
\mathcal{K},w  \sim_\omega \mathcal{K}',w' 
\; \Leftrightarrow \;
\mathcal{K},w  \eqml \mathcal{K}',w'.
\]

We here establish an inquisitive version of this Hennessy-Milner phenomenon,
which technically works with $\omega$-saturated relational pseudo-models and
our standard translation. It turns out that the passage through pseudo-models
is absolutely crucial for the argument (cf.~Observation~\ref{omegadownarrowclashobs}).

\medskip
In what follows we could consider signatures of arbitrary infinite cardinalities,
i.e.\ with arbitrary index sets $I$ for the family of basic propositions
$(p_i)_{i \in I}$; but to simplify the formal account, and without any essential
loss of generality, we stick to at most countably infinite~$I$. Consider
$\INQML$-types of worlds of inquisitive (pseudo-)models: for a world $w \in W$ in
the (pseudo-)model $\M$, let
$\rho_w:=\{\phi\in\INQML\colon\M,w\models\phi\}$
be its $\INQML$-type. Even if the overall signature is infinite, 
any individual $\INQML$- or $\FO$-formula can only refer to finitely many basic
propositions (or their relational counterparts). In this sense we have 
$\INQML = \bigcup_k \INQML_k$ for its fragments $\INQML_k$ that use just the first
$k$ basic propositions from $I_k \subset I$, for $k \in \N$. The $\INQML$-type of
a world in a (relational) pseudo-model therefore always is the union of its
restrictions to finite sub-signatures, which are types in corresponding reducts.

For bisimulation types, i.e.\ equivalence classes w.r.t.\ $\sim$ or $\sim_n$,
however, the situation is different. If we denote as $\sim_{n,k}$ the notion of
$n$-bisimilarity of the $I_k$-reducts of models, then even at the level of $n=1$,
all the $\sim_{1,k}$-types do not determine the $\sim_1$-type of a world $w$, as
the former do e.g.\ not determine whether $w$ has a Kripke-successor in which
infinitely many basic propositions are satisfied simultaneously.

In the following we want to consider standard translations of $\INQML$-formulae on
world-pointed relational inquisitive models and therefore use the
world-version of the standard translation from~\Cref{worldstandardtranslation}
whenever this is appropriate.
With a world $w$ in some relational pseudo-model we associate 
its complete $\INQML$-type $\rho_w$ with a partial first-order type
$\rho^\ast_w(x)$
consisting of the standard translations of all formulae in $\rho_w$.
So $\rho_w$ is determined by (the standard translations of) the characteristic
formulae $\chi^{n,k}_w\in\INQML_k$ for the $\sim_n$-types of $w$ in
the $I_k$-reducts,
i.e.\ the $\sim_{n,k}$-types, for all $n,k \in \N$, by Theorem~\ref{EFthm}.

From now on, for reasons of simplicity, we identify $\INQML$-formulae like
$\chi^{n,k}_v$ with their standard translation in $\FO$,
which is expressed as $\chi^{n,k}_v(x)$ in a first-order variable
of the first sort, or as $\chi^{n,k}_v(\lambda)$ in a first-order
variable of the second sort in case we consider state-pointed relational
inquisitive models.
 
Recall from classical model theory 
(see~e.g.~\cite{HodgesMT}) that a first-order structure is
\emph{$\omega$-saturated} if it realises every first-order type with 
finitely many parameters.
Generally, a (partial) \emph{first-order type} of a $\sigma$-structure $\mathcal{A}$ 
is a set of $\FO(\sigma)$-formulae $\Phi$ in some tuple of 
free variables that is consistent with the complete 
$\FO(\sigma)$-theory of $\mathcal{A}$; a tuple of elements of
$A$ \emph{realises} $\Phi$ if all $\phi \in \Phi$ are satisfied 
by this tuple.
A (partial) type with parameters~$\mathbf{a}$ from $A$ is a (partial) type 
of the expansion of $\mathcal{A}$ with new constant
symbols for the $a \in \mathbf{a}$, so that $\Phi_{\mathbf{a}}$ can
specify first-order properties in relation to these parameters in
$(\mathcal{A},\mathbf{a})$.\footnote{
We here look at partial $\FO$-types in the  
two-sorted setting of relational pseudo-models. These may specify 
properties of worlds (in free variables of the first sort) as well as
of information states (in free variables of the second sort), and
could involve parameters from both sorts. As it turns out, we 
actually just need to consider partial types of individual elements
(of either sort) with single parameters from the other sort for our purposes.}

Classical chain constructions based on $\FO$ compactness establish that
any first-order structure admits elementary extensions that are
$\omega$-saturated. Compactness is also at the heart of the straightforward
argument for the following fact, to be used in the proof of our main result
of this section.

In the following, we say that an element $s\in\Po(W)$ is \emph{represented} in the
second sort of a relational pseudo-model if $s\in S\subseteq\Po(W)$.

\begin{lemma}
\label{omegatyperem}
Let $s \in S$ be represented in the second sort of an 
$\omega$-saturated relational pseudo-model $\Mod$,
$\rho_v(x)$ as above (not necessarily realised in $\Mod$). 
Then $\rho_v(x)$ is realised by a world in $s$ in $\Mod$ iff 
\[
\Mod,s \models \exists x \left( x\!\in\! s \wedge \chi^{n,k}_v(x) \right)
\; \mbox{ for all $k,n \in \N$.}
\]
\end{lemma}

This lemma
implies in particular that $\omega$-saturation 
is at odds with full downward closure.
This is the reason why our inquisitive version of the Hennessy--Milner
correspondence must be based on $\omega$-saturated relational 
pseudo-models rather than relational encodings of proper models.

\begin{observation}
\label{omegadownarrowclashobs}
In general, $\omega$-saturation is 
incompatible with the $\downarrow$-closure condition on information
states.
It requires \emph{all}
information states $s$ that are represented in the second sort to be 
closed under ``limits'' of partial $\sim$-types of worlds $w \in s$.
For instance, one cannot have the analogue of 
the information state ``all Kripke-paths from worlds $w\in s$ are finite'' 
without imposing a uniform finite bound on the lengths of all these paths.
\end{observation}

Indeed, an information state whose worlds admit arbitrarily long but
just finite Kripke-paths (L\"ob condition) 
cannot be represented in an $\omega$-saturated pseudo-model, because
any such state would also have to contain a world with an infinite Kripke-path.
This follows as $\omega$-saturation would inductively always provide a next
Kripke-successor that still has Kripke-paths of unbounded finite lengths in front of it.

\medskip
We define an auxiliary
equivalence relation between state-pointed pseudo-models as follows:
$\M,s$ and $\M',s'$ are \emph{bulk-equivalent},
\[
(\dagger) \quad \M,s \eqinqmlbulk\M',s',
\]
if for every $w\in s$ there is some $w'\in s'$ such that
$\M,w  \eqinqml \M',w'$, and vice versa.
In general, $\eqinqmlbulk$
is a proper strengthening of $\eqinqml$; the two do coincide, however,
for $\omega$-saturated pseudo-models and for information states that are
represented in the second sort. In that situation,
ordinary $\INQML$-equivalence reflects the natural relationship between state- and
world-based inquisitive bisimulation: state-equivalence
$\M,s \eqinqml \M',s'$
as a flat lifting of world-equivalence,
in the spirit of our bulk equivalence.

\begin{remark}
\label{IMLQbisimflatnesslem}
For state-pointed pseudo-models $\M,s$ and $\M',s'$ that 
stem from $\omega$-saturated relational state-pointed pseudo-models
$\Mod,s$ and $\Mod',s'$ with $s$ and $s'$ being represented in the second sort of
$\Mod$ and $\Mod'$, respectively, the following are equivalent:
\bre
\item $\M\!\downarrow,s\eqinqml \M'\!\downarrow,s'$;
\item $\M\!\downarrow,s \eqinqmlbulk \M'\!\downarrow,s'$,
i.e., for every $w\in s$ there is some $w'\in s'$ such that $\M,w  \eqinqml
\M',w'$, and vice versa.
\ere
\end{remark}

\begin{proof}
Note that references to inquisitive closures $\M\!\downarrow,s$ 
could be replaced by $\M,s$ in both~(i) and~(ii), as generally 
$\M,s \eqinqml \M\!\downarrow,s$ by~Proposition~\ref{pseudoclosuresame}.

The implication (ii)~$\Rightarrow$~(i) is obvious in light of the definition of the
semantics of $\INQML$ (over pseudo-models) in combination with its 
graded flatness. 
This implication does not require $\omega$-saturation.

For~(i)~$\Rightarrow$~(ii), let $\M,s \eqinqml \M',s'$ and consider
the claim in~(ii) e.g.\ for $w \in s$. On the basis of $\omega$-saturation and
Lemma~\ref{omegatyperem}, it suffices to show that for all $n,k \in
\N$, there is some $w' \in s'$ s.t.\ $\M',w' \models \chi_w^{n,k}$. 
If this were not the case for some $n,k$, 
then $\M',s' \models \neg \chi_w^{n,k}$, which is impossible since 
obviously not $\M,s \models \neg \chi_w^{n,k}$ (cf.\ remarks on the
truth-conditional, flat nature of $\chi^{n,k}_w$ in connection with
Theorem~\ref{EFthm} above).
\end{proof}

\begin{lemma}
\label{approxtosimlem}
For pairs of state-pointed pseudo-models $\M,s$ and $\M',s'$
stemming from $\omega$-saturated relational pseudo-models
$\Mod$ and $\Mod'$:
\[
\M \!\downarrow,s \eqinqmlbulk \M' \!\downarrow,s' 
\; \Rightarrow \;
\M \!\downarrow,s  \sim \M'\!\downarrow,s'.
\]
\end{lemma}

Note that the information states $s \subset W$ and $s' \subset W'$
in the premise of the claim of the lemma are arbitrary information
states and need not themselves be represented in the second sort of
the underlying relational pseudo-models $\Mod$ and $\Mod'$.
However, by~\Cref{IMLQbisimflatnesslem}, \Cref{approxtosimlem} implies
the same conclusion for states $s$ and $s'$ that are represented in the
second sort of $\Mod$ and $\Mod'$, respectively, under the weaker
assumption that $\M \!\downarrow,s \eqinqml \M' \!\downarrow,s' $.

\begin{proof}
The idea is to represent or simulate the infinite bisimulation game
from~\cite{CiardelliOtto} on $\M\!\downarrow;\M'\!\downarrow$ in terms
of the underlying relational pseudo-models $\Mod;\Mod'$. 
The challenge lies in dealing with those information states that may
essentially occur in the bisimulation game without being represented
in the second sort of the underlying relational pseudo-models.
\par
For the claim of the lemma it suffices to show that player~$\mathbf{II}$
can maintain~$\eqinqmlbulk$ through a single round in the following sense:
starting from a state position
$\bigl(\M\!\downarrow,s;\M'\!\downarrow,s'\bigr)$ with $\M\!\downarrow,s\eqinqmlbulk\M'\!\downarrow,s'$, 
player~$\mathbf{II}$ can respond to any world challenge from player~$\mathbf{I}$ to
get to a world position
$\bigl(\M\!\downarrow,w;\M'\!\downarrow,w'\bigr)$ such that $\M\!\downarrow,w\eqinqml\M'\!\downarrow,w'$,
and from there respond to any state challenge from player~$\mathbf{I}$ to
get again to a state position
$\bigl(\M\!\downarrow,t;\M'\!\downarrow,t'\bigr)$ such that $\M\!\downarrow,t\eqinqmlbulk\M'\!\downarrow,t'$.
\par
Let $\M\!\downarrow,s \eqinqmlbulk \M'\!\downarrow,s'$ and assume
w.l.o.g.\ that player~$\mathbf{I}$ picks an element $w \in s$;
then by the definition of $\eqinqmlbulk$ at $(\dagger)$, 
player~$\mathbf{II}$ can respond with some $w' \in s'$
with $\M\!\downarrow,w\eqinqml\M'\!\downarrow,w'$.
If player~$\mathbf{I}$ now picks (again w.l.o.g.)\
an information state $t\in\Sigma\!\downarrow\!(w)$ in $\M\!\downarrow$,
we need player~$\mathbf{II}$ to respond with 
some $t' \in \Sigma'\!\downarrow\!(w')$ in $\M'\!\downarrow$
such that $\M\!\downarrow,t \eqinqmlbulk \M'\!\downarrow,t'$. 
\par
For this consider the partial $\FO$-type describing the essential 
properties of $t$ (in variable $\lambda$ of the second sort) in
$\Mod$ (with parameter~$w$):\footnote{
This partial type is an approximation of the bisimulation type of $t$,
in two senses: it only stipulates positive assertions about bisimulation
types of worlds~$v$ in~$t$; and each such $\sim$-type of a world~$v$
can only be represented by its $\sim_{n,k}$-approximants.}
\[
\Phi_w(\lambda) := \{ E w\lambda \} \cup 
\left\{ \exists x \left(  x \in \lambda \wedge \chi_v^{n,k} (x) \right)\colon 
n,k \in \N, v \in t \;\right\},
\]
where $\chi_v^{n,k}(x)$ is (the standard translation of) the
$\INQML$-formula that characterises the $\sim_{n,k}$-type 
of $\M\!\downarrow, v$.
We want to show that $\Phi_{w'}(\lambda)$ (with parameter~$w'$) is finitely satisfiable in $\Mod'$ and hence a partial type of
$\Mod'$. It suffices to show that $\Mod',w'$ satisfies every formula of the form
\[ \phi(x)=
\exists \lambda 
\Bigl(  E x \lambda \wedge
\bigwedge_{i\leq m}
\exists y
\bigl( y\!\in\! \lambda \wedge \chi_{v_i}^{n,k}(y)\bigr)\Bigr),
\]
for $n,m,k \in \N$ and $v_1,\ldots,v_m \in t$. For this claim we can use
the fact that $\M,w \eqinqml \M',w'$, which implies that
$\M\!\downarrow,w \sim_{n+1,k}\M'\!\downarrow,w'$.
So for the state $t\in\Sigma\!\downarrow\!(w)$ there is some $t^*\in \Sigma'\!\downarrow\!(w')$
in $\M'\!\downarrow$ with $\M\!\downarrow,t \sim_{n,k} \M'\!\downarrow,t^*$.
Let $\hat{t}'\in \Sigma'(w')$ be an information state in the pseudo-model $\M'$ 
that is represented in the second sort of $\Mod'$, such that $t^*\subset\hat{t}'$.
Then for every $v\in t$ there is some $v'\in t^*\subset\hat{t}'$ with
$\M\!\downarrow,v \sim_{n,k} \M'\!\downarrow,v'$, hence with
$\M'\!\downarrow,v'\models\chi_{v}^{n,k}(y)$. So $\Mod',w'\models\phi(x)$,
with $\hat{t}'$ as an existential witness for $\lambda$. Then
$\Phi_{w'}(\lambda)$ is consistent with the $\FO$-theory of $\Mod',w'$
(i.e.\ a partial type with parameter $w'$ in $\Mod'$).
\par
By $\omega$-saturation of~$\Mod'$ there is a
realisation $\tilde{t}'$ of $\Phi_{w'}(\lambda)$. So
(by~\Cref{omegatyperem} and using $\omega$-saturation again)
there is,  for every world $v\in t$, some world $v'\in \tilde{t}'$ that realises
its $\INQML$-type $\rho_v$, i.e.\ such that
$\M\!\downarrow,v \eqinqml \M'\!\downarrow,v'$. Now 
\[
t':=\{v'\in\tilde{t}'\colon\exists v\in t\text{ with }\M\!\downarrow,v
\eqinqml \M'\!\downarrow,v'\} \subset \tilde{t}'
\]
is an information state in $\Sigma'\!\downarrow\!(w')$, as
$\tilde{t}'\in\Sigma'(w')$ in $\M'$ ($\tilde{t}' \in E[w']$ in $\Mod'$).
It follows that $t' \in \Sigma'\!\downarrow\!(w')$ is an appropriate
response to match $t \in \Sigma\!\downarrow\!(w)$. Indeed, 
for every world $v \in t$  there is a world $v' \in t'$ and vice versa, with 
$\M\!\downarrow,v \eqinqml \M'\!\downarrow,v'$. So
we have $\M\!\downarrow,t \eqinqmlbulk \M'\!\downarrow,t'$, and $t'$
is as desired. 
\end{proof}

\Cref{approxtosimlem} does give rise to a nice class of 
models with a Hennessy--Milner property for $\INQML$.

\begin{corollary}[a Hennessy--Milner class for $\INQML$]
\label{hmclass}
\mbox{}\\
If $\M = \M(\Mod)$ and $\M'= \M(\Mod')$ arise as pseudo-models
from $\omega$-saturated relational pseudo-models $\Mod$ and $\Mod'$,
then their $\downarrow$-closures, which are proper models, satisfy the
following Hennessy--Milner property for bulk-equivalent\footnote{See
$(\dagger)$ for the definition of $\eqinqmlbulk$ as the
  flat lifting of $\eqinqml$ and see also Remark~\ref{IMLQbisimflatnesslem}.}
states:
\[
\M\!\downarrow, s\eqinqmlbulk
\M'\!\downarrow, s'
\;\; \Leftrightarrow \;\;
\M\!\downarrow, s \;\sim\; 
\M'\!\downarrow, s'.
\]
Similarly, for worlds:
\[
\M\!\downarrow, w\eqinqml
\M'\!\downarrow, w'
\;\; \Leftrightarrow \;\;
\M\!\downarrow, w \;\sim\; 
\M'\!\downarrow, w'.
\]
\end{corollary}

\begin{proof}
The left-to-right implication is from~\Cref{approxtosimlem};
the converse follows from preservation of $\INQML$ 
under inquisitive bisimulation.
\end{proof}

As further corollaries we obtain characterisation results
for $\INQML$ as a fragment of $\FO$, in the style of 
a van~Benthem theorem. The relevant notions of inquisitive bisimulation 
equivalence can be naturally transferred from models to their
relational encodings as in~\cite{CiardelliOtto}.

\begin{definition}
\label{inqbismdef}
Inquisitive bisimulation equivalence for world-pointed
relational models,  $\Mod,w \sim \Mod',w'$, is defined
by direct transfer from the underlying inquisitive models, as
$\Mod,w \sim \Mod',w' \; :\Leftrightarrow \; 
\M(\Mod), w\sim\M(\Mod'),w'$; and analogously for the 
state-pointed case.
\end{definition}

The following corollary crucially
differs from the
characterisation obtained along very different routes 
in~\cite{CiardelliOttoLong}.
On the one hand, the underlying elementary class of pseudo-models, which
our new characterisation refers to, is considerably wider than the
non-elementary class of relational encodings of proper models. On the
other hand, it therefore significantly involves preservation
assumptions that bridge the gap between pseudo-models and proper models.
Involving the natural extension of the semantics of $\INQML$ to this
wider class of pseudo-models, it also enriches the analysis of the 
relation between $\INQML$ and $\FO$ in another direction.

We first state and sketch the proof for the simpler version of this
characterisation for world properties; the more general characterisation of
state properties, with persistency as an additional 
semantic constraint, is then treated in Corollary~\ref{vBcor_statepointed}.

Compare Definition~\ref{inqclosure} for the notion of \emph{inquisitive closure},
which involves downward closure w.r.t.\ information states in inquisitive assignments,  
in the passage from $\Sigma$ to $\Sigma\!\downarrow$ or 
the analogous passage in the relational encoding of $\Sigma$ in terms of $E$.

\begin{corollary}
\label{vBcor}
The following are equivalent for $\phi= \phi(x) \in \FO$ (in the vocabulary
of two-sorted relational (pseudo-)models and in a single free
variable~$x$ of the first sort):
\bre
\item[(i)]
\bre
\item[--]
$\phi$ and $\neg \phi$ are preserved under passage to inquisitive closures 
over the elementary class of all world-pointed relational pseudo-models, and
\item[--]
$\phi$ is preserved under inquisitive bisimulation $\sim$ over the
(non-ele\-men\-tary) class of all relational encodings of world-pointed models;
\ere
\item[(ii)]
$\phi$ is logically equivalent over the class of all world-pointed relational
pseudo-models to (the standard translation of) some formula $\psi \in \INQML$.
\ere 
\end{corollary}

\begin{proof}
For (i) $\Rightarrow$ (ii) we argue for the contrapositive. If $\phi$ is not
equivalent to any $\INQML$-formula, then by compactness, there are $\omega$-saturated
relational pseudo-models $\Mod,w$ and $\Mod',w'$ such that $\Mod,w \eqinqml \Mod',w'$
but $\Mod,w \models \phi$ while $\Mod',w' \models \neg \phi$. By the previous corollary
$\Mod\!\downarrow,w \sim \Mod'\!\downarrow,w'$, and as $\Mod,w\models\phi$ while
$\Mod',w'\models\neg\phi$, $\phi$ and $\neg \phi$ cannot be preserved under 
$\sim$ and under passage to inquisitive closures as stipulated in~(i).
\end{proof}

\begin{remark}
\label{bisimforpseudorem}
The combined preservation condition in~(i) above can equivalently be
replaced by a single preservation condition w.r.t.\ the natural
extension of inquisitive bisimulation to (relational encodings of) 
pseudo-models. For this one could just define, in extension of 
Definition~\ref{inqbismdef} and, e.g.~for world-pointed pseudo-models,  
$\Mod,w \sim \Mod',w' \; :\Leftrightarrow\;
\M(\Mod)\!\downarrow,w\sim\M(\Mod')\!\downarrow,w'$.
\end{remark}

For the following we adapt the notion of inquisitive closure 
(from Definition~\ref{inqclosure}) to state-pointed relational 
pseudo-models so as to include all subsets of
the distinguished information state in the second sort:\footnote{Representation
of all subsets of the distinguished information state in the second
sort is essential if we want to capture the idea of persistency in the relational setting.}
the inquisitive closure $(\Mod,s)\!\downarrow$ 
augments the second sort of the relational encoding of 
$\M(\Mod)\!\downarrow$ by representations of all subsets of $s$ 
in the second sort in as far as they are not already covered as 
subsets of information states in inquisitive assignments.

In the following, we invoke 
the set of all \emph{boolean combinations of (standard translations)
  of $\INQML$-formulae}, 
meaning the closure of all standard translations of $\INQML$-formulae
under the classical $\FO$-connectives $\neg$, $\wedge$ and $\vee$.

\begin{lemma}
\label{persistentBClem}
Over the class of inquisitive closures of state-pointed relational
pseudo-models with non-empty distinguished states:
if a boolean combination of (standard translations of)
$\INQML$-formulae is persistent,
then it is logically equivalent
to (the standard translation of)  a single $\INQML$-formula.
\end{lemma}

\begin{proof}
For first-order $\psi(\lambda)$ in a free variable $\lambda$ of the
second sort, let 
\[
\psi^\downarrow(\lambda) := 
\forall\mu\bigl(\mu\subset\lambda\rightarrow\psi(\mu)\bigr),
\]
where $\mu \subset \lambda$ stands for the 
natural rendering in our $2$-sorted first-order
framework, which is literally adequate in inquisitive closures of
state-pointed relational pseudo-models. Then $\psi$ is persistent
if, and only if, $\psi \equiv \psi^\downarrow$ over this class. It is
easy to check that $\downarrow$ commutes with conjunction in the sense
that
\[
(\psi_1 \wedge \psi_2)^\downarrow(\lambda) 
\equiv \psi_1^\downarrow(\lambda) 
\wedge 
\psi_2^\downarrow(\lambda).
\]
Consider now a persistent formula $\psi (\lambda)$ that is a boolean
combination of standard translations, w.l.o.g.\ in conjunctive normal
form. As $\downarrow$ commutes with conjunction, it suffices to
show that the $\downarrow$-version of each conjunct
of the conjunctive normal form satisfies the claim of the lemma.
And as the standard translation commutes with conjunction and disjunction
(i.e.\ $\store$ in $\INQML$ vs.\ $\vee$ in $\FO$),
it suffices to prove expressibility of $\psi^\downarrow$ as a standard translation 
of an $\INQML$-formula for a simple disjunction of the form 
\[
\psi(\lambda)= \neg \phi_1^\ast(\lambda) \vee \phi_2^\ast(\lambda)
\]
for $\phi_1,\phi_2 \in \INQML$ with standard translations
$\phi_i^\ast(\lambda)$. We here think of
$\phi_1$ as a (possibly empty) conjunction and of $\phi_2$ as a
(possibly empty) disjunction. The degenerate cases, in which
$\psi$ has trivial positive or negative contributions, correspond to
(classical) $\bot$ or $\top$ in the places of $\phi_2^\ast$ or
$\phi_1^\ast$, respectively. And as classical $\bot$ agrees with
the standard translation of inquisitive $\bot$ only for non-empty
states, it is important that we aim for equivalence w.r.t.\
non-empty states only. But in all those scenarios we find that
$\psi^\downarrow(\lambda)$ is equivalent to the standard translation of the
inquisitive implication $\phi_1 \rightarrow \phi_2 \in \INQML$:
\[
\psi^\downarrow(\lambda)\equiv 
(\phi_1 \rightarrow \phi_2)^\ast(\lambda)
\]
over the class of inquisitive closures of
state-pointed relational pseudo-models with
non-empty distinguished states.
\end{proof}

\begin{corollary}
\label{vBcor_statepointed}
The following are equivalent for $\phi = \phi(\lambda) \in \FO$ (in the vocabulary
of two-sorted relational (pseudo-)models and in a single free
variable~$\lambda$ of the second sort):
\bre
\item
\bae
\item[--]
$\phi$ and $\neg \phi$ are preserved under passage to the inquisitive
closure of state-pointed relational pseudo-models,
\item[--]
$\phi$  is preserved under inquisitive bisimulation $\sim$ over the
(non-ele\-men\-tary) class of all relational encodings of state-pointed models, 
\item[--]
$\phi$ is persistent over the (non-elementary) class of inquisitive
closures of state-pointed relational pseudo-models;
\eae
\item
$\phi$ is logically equivalent to (the standard translation of) some
formula $\psi \in \INQML$ over the class of all relational pseudo-models
with non-empty distinguished information state.
\ere
\end{corollary}

\begin{proof}
As in the proof of~\Cref{vBcor}, 
the implication (ii) $\Rightarrow$ (i) is obvious, and 
we argue indirectly for (i)~$\Rightarrow$ (ii):
we show that, if $\phi$ is \emph{not} equivalent,  
w.r.t.\ non-empty information states,
to any \emph{boolean combination} of $\INQML$-formulae
then it cannot satisfy the preservation and invariance 
conditions in~(i). This then implies~(ii) with~\Cref{persistentBClem}.
The crucial claim for that is the following. 
If $\phi(\lambda)$ is not equivalent to any boolean combination of standard translations
of $\INQML$-formulae, then there is a set $\Psi = \Psi(\lambda)$ 
of first-order formulae in 
the single free variable $\lambda$ of the second sort, such that
\bae
\item
for all standard translations $\psi(\lambda)$ of $\INQML$-formulae, 
either $\psi \in \Psi$ or $\neg \psi \in \Psi$,
\item
$\Psi \cup \{ \phi \}$ and $\Psi \cup \{ \neg \phi \}$ are both
satisfiable
by non-empty information states in
relational pseudo-models.
\eae
Condition~(a) says that $\Psi$ is ``complete'' w.r.t.\
the set of all standard translations of $\INQML$-formulae; (b) guarantees
that there are $\omega$-saturated relational pseudo-models 
$\Mod,s$ and $\Mod',s'$ (models of $\Psi$
with states $s,s'$ that are represented)
such that
$\Mod,s \equiv_{\INQML} \Mod',s'$, whence also 
$(\Mod,s)\!\downarrow,s \equiv_{\INQML} (\Mod',s')\!\downarrow,s'$, 
but $\Mod,s \models \phi$ while
$\Mod',s' \models \neg \phi$. This
contradicts the preservation and invariance
as stated in~(i):
preservation of~$\phi$ and $\neg\phi$ under passage to closures
implies
$(\Mod,s)\!\downarrow,s \models \phi$ and 
$(\Mod',s')\!\downarrow,s' \models \neg \phi$. Moreover,
$\Mod,s \equiv_{\INQML} \Mod',s'$ implies
$(\Mod,s)\!\downarrow,s \sim (\Mod',s')\!\downarrow,s'$
with~\Cref{IMLQbisimflatnesslem} and~\Cref{hmclass};
but this contradicts preservation of $\phi$
under inquisitive bisimulation.

The set $\Psi$ is obtained by a compactness argument, which is
more involved than the analogue used in the world-pointed case of
Corollary~\ref{vBcor}. The present argument crucially involves
the additional persistency assumption in~(i). Assuming w.l.o.g.\ 
a finite signature we may enumerate all
standard translations of $\INQML$-formulae as $(\psi_i(\lambda))_{i
  \in \N}$,
and obtain the desired $\Psi$ as the union of a chain of finite
sets $\Psi_i$ as follows.
Starting from
$\Psi_0 := \emptyset$
we inductively augment 
$\Psi_i$ by either $\psi_i$ or $\neg \psi_i$ to obtain
$\Psi_{i+1}$. In the inductive process we maintain, as an invariant, 
the condition that for all boolean combinations of standard
translations $\psi(\lambda)$ of
$\INQML$-formulae, and over the elementary class of all
state-pointed relational pseudo-models
with a non-empty distinguished information state,\footnote{In the
remainder of this argument we always assume to be
  working over this elementary class: even where this is not stated
  explicitly, all first-order consequence relations are to be read in
  this context.} 
\[
(\ast) \quad
\Psi_i(\lambda) \not\models \phi(\lambda) \leftrightarrow 
\psi(\lambda),
\]
i.e.\ that $\phi$ remains inequivalent to all boolean combinations of
standard translations over the class of all state-pointed relational
pseudo-models of $\Psi_i$ with non-empty distinguished state.
Then $\Psi = \bigcup_i \Psi_i$ meets requirements~(a) and~(b); 
indeed, with compactness, (b)~is covered by~$(\ast)$ for (the standard
translations of) $\bot$ or $\top$ in the r\^ole of $\psi$. If $\Psi \cup \{ \phi \}$
were not satisfiable in a non-empty information state, then 
$\Psi(\lambda)  \models \phi(\lambda) \leftrightarrow \bot$
over the class of relational pseudo-models with non-empty information
states (and over this class, classical $\bot$ is equivalent to the standard
translation of inquisitive $\bot$); similarly for $\Psi \cup \{ \neg \phi \}$
and equivalence with $\top$.
 
The initial set $\Psi_0 = \emptyset$ satisfies
condition~$(\ast)$ by assumption. It remains to argue that~$(\ast)$
can be maintained in at least one of the augmentations
$\Psi_i \cup \{ \psi_i \}$ or $\Psi_i \cup \{ \neg \psi_i \}$.
Assume this were not the case, i.e.\ that, for some boolean combinations $\xi_0,\xi_1$
of standard translations of $\INQML$-formulae,
\[
\Psi_i \cup \{ \psi_i \}
\models \phi \leftrightarrow \xi_0 \quad \mbox{ and } \quad
\Psi_i \cup \{ \neg \psi_i \}
\models \phi \leftrightarrow \xi_1.
\]
But then
\[\Psi_i \models 
\phi \leftrightarrow \bigl((\psi_i \wedge \xi_0)
\vee (\neg \psi_i \wedge \xi_1) \bigr)\]
contradicts the assumption that~$(\ast)$ was satisfied in the previous step.
\end{proof}

\section{Conclusion}
\label{conclusionsec}

Our new standard translation combines several advantages: it is
defined directly and inductively, and it works with minimal 
requirements on the corresponding class of relational models. 
The basic idea of the standard translation and its use in a  
model-theoretic compactness proof requires a non-trivial adaptation of the 
well-known treatment of basic modal logic $\ML$, due to the inherent
two-sortedness of the relational representations of the intended
models for inquisitive modal logic $\INQML$; the notion of graded flatness
plays a key r\^ole in taming the salient second-order features. Other than
for $\ML$, the relational counterparts of the intended models for
$\INQML$ do not form an elementary class. Nevertheless, 
with our standard translation we could give a purely model-theoretic
compactness proof for $\INQML$ over the class of all relational 
inquisitive models, as well as over the class of all relational
pseudo-models. The usefulness of model-theoretic constructions
based on pseudo-models has further been exemplified  with a novel
Hennessy--Milner theorem for $\INQML$, which crucially involves the
passage through pseudo-models and again relies on our standard translation. 
As corollaries we obtained a world-pointed and a state-pointed variant of
a van Benthem style theorem that characterises $\INQML$ in 
relation to $\FO$ over the elementary class of all pseudo-models. 
Overall our findings seem to recommend the class of pseudo-models as a
useful model-theoretic backdrop for the study of inquisitive modal
logic, if not even as a suitable alternative for the class of proper models 
for some specific aspects. Further investigations in this vein could
focus on model-theoretic approaches, via pseudo-models, to special 
classes of inquisitive frames or, for instance, a Lindstr\"om 
characterisation of inquisitive modal logic. 

\paragraph*{Acknowledgement.} We would like to thank the anonymous
referees for suggesting that we substantially extend the
scope of our discussion of inquisitive pseudo-models 
beyond the primary application (viz.~the 
model-theoretic compactness proof of Section~\ref{compactsec}).
This encouraged us to include more recent results involving
saturation and the analysis of bisimulation in the context of 
pseudo-models, as treated in the new Section~\ref{HMsec}.
And that new section, too, profited from the referee's critical reading.

\renewcommand\bibname{References}
\bibliographystyle{abbrv}
\bibliography{ref}

\end{document}